\documentclass[11pt]{article}
\usepackage{amssymb,amsthm,amsmath}
\usepackage[utf8]{inputenc}

\usepackage{geometry}
\usepackage{amssymb}
\usepackage{verbatim}
\usepackage{amsmath}
\usepackage[T1]{fontenc}	
\usepackage[utf8]{inputenc}
\usepackage{enumerate}
\usepackage{amsfonts}
\usepackage{graphicx}
\newcommand{\call}[4]{\int_{#1}^{#2} {#3} \; \textrm{d} {#4}}
\newcommand{\mb}[1]{\mathbb{{#1}}}
\newcommand{\mc}[1]{\mathcal{{#1}}}
\newcommand{\1}{\mathbf{{1}}}
\newcommand{\e}{\varepsilon}
\newcommand{\po}[2]{\frac{\textrm{d} #1}{\textrm{d} #2}}

\newcommand{\vp}{\varphi}

\newcommand{\norma}[2]{\Vert #1 \Vert_{#2}}

\newcommand{\scal}[2]{\left\langle #1, #2 \right\rangle}

\DeclareMathOperator{\vol}{vol}
\usepackage[colorinlistoftodos]{todonotes}
\usepackage{color}

\theoremstyle{definition}

\newtheorem{prop}{Proposition}
\newtheorem{rem}{Remark}
\newtheorem{thm}{Theorem}

\newtheorem{cor}{Corollary}
\newtheorem{lem}{Lemma}
\newtheorem{ex}{Example}
\newtheorem{defi}{Definition}

\newtheorem*{que*}{Question}

\title{On the Brunn-Minkowski inequality for general measures with applications to new isoperimetric-type inequalities}

\author{Galyna Livshyts\thanks{ supported in part by the U.S. National Science
Foundation Grant DMS-1101636}  \thanks{supported in part by the Institute for Mathematics and its Applications with funds provided by the National Science Foundation}, Arnaud Marsiglietti\footnotemark[2], Piotr Nayar \thanks{supported in part by NCN grant DEC-2012/05/B/ST1/00412} \footnotemark[2], Artem Zvavitch\footnotemark[1]}
\date{\today}

\begin{document}

\newgeometry{tmargin=2.2cm, bmargin=2.4cm, lmargin=2.2cm, rmargin=2.2cm}

\maketitle

\begin{abstract}
In this paper we present new versions of the classical Brunn-Minkowski inequality for different classes of measures and sets. We show that the inequality
\[
	\mu(\lambda A + (1-\lambda)B)^{1/n} \geq \lambda \mu(A)^{1/n} + (1-\lambda)\mu(B)^{1/n}
\]
holds true for an unconditional product measure $\mu$ with decreasing density and a pair of unconditional convex bodies $A,B \subset \mb{R}^n$. We also show that the above inequality is true for any unconditional $\log$-concave measure $\mu$ and unconditional convex bodies $A,B \subset \mb{R}^n$. Finally, we prove that the inequality is true for
a symmetric $\log$-concave measure $\mu$ and a pair of symmetric convex sets $A,B \subset \mb{R}^2$, which, in particular,  settles two-dimensional case of the conjecture for Gaussian measure proposed in \cite{GZ}.

In addition, we deduce the $1/n$-concavity of the parallel volume $t \mapsto \mu(A+tB)$, Brunn's type theorem and certain analogues of Minkowski first inequality.
\end{abstract}



\vspace{1cm}

\section{Introduction}

The classical Brunn-Minkowski inequality states that for any two non-empty compact sets $A,B \subset \mb{R}^n$ and any $\lambda \in [0,1]$ we have
\begin{equation}\label{clasBM}
	\vol_n(\lambda A + (1-\lambda)B)^{1/n} \geq \lambda \vol_n(A)^{1/n} + (1-\lambda)\vol_n(B)^{1/n},
\end{equation}
with equality if and only if $B=aA+b$, where $a >0$ and  $b \in \mb{R}^n$. Here $\vol_n$ stands for the Lebesgue measure on $\mb{R}^n$ and
\[
A + B = \{a + b : a\in A, b \in B\}
\]
is the Minkowski sum of $A$ and $B$. Due to homogeneity of the volume, this inequality is equivalent to $\vol_n(A+B)^{1/n} \geq \vol_n(A)^{1/n}+\vol_n(B)^{1/n}$. The Brunn-Minkowski inequality turns out to be a powerful tool. In particular, it implies the classical isoperimetric inequality: for any compact set $A \subset \mb{R}^n$ we have $\vol_n(A_t) \geq \vol_n(B_t)$, $t \geq 0$, where $B$ is a Euclidean ball satisfying $\vol_n(A)=\vol_n(B)$ and $A_t$ stands for the $t$-enlargement of $A$, i.e., $A_t=A+tB_2^n$, where $B_2^n$ is the unit Euclidean ball, $B_2^n=\{x: |x|=1\}$. To see this it is enough to observe that
\[
	\vol_n(A+tB_2^n)^{1/n} \geq  \vol_n(A)^{1/n} +  \vol_n(tB_2^n)^{1/n}  =  \vol_n(B)^{1/n} +  \vol_n(tB_2^n)^{1/n} =\vol_n(B+tB_2^n)^{1/n}.
\]
Taking $t \to 0^+$ one gets a more familiar form of isoperimetry: among all sets with fixed volume the surface area
\[
	\vol_n^+(\partial A) = \liminf_{t \to 0^+} \frac{\vol_n(A + tB_2^n) - \vol_n(A)}{t}
\]
is minimized in the case of the Euclidean ball. We refer to \cite{G1} for more information on Brunn-Minkowski-type inequalities.

Using the inequality between means one gets an a priori weaker dimension free form of \eqref{clasBM}, namely
\begin{equation}\label{wBM}
	\vol_n(\lambda A + (1-\lambda)B) \geq \vol_n(A)^\lambda \vol_n(B)^{1-\lambda}.
\end{equation}
In fact \eqref{wBM} and \eqref{clasBM} are equivalent. To see this one has to take $\tilde{A}=A/\vol_n(A)^{1/n}$, $\tilde{B}=B/\vol_n(B)^{1/n}$ and $\tilde{\lambda}=\lambda \vol_n(A)^{1/n}/(\lambda \vol_n(A)^{1/n} + (1-\lambda) \vol_n(B)^{1/n})$ in \eqref{wBM}. This phenomenon is a consequence of homogeneity of the Lebesgue measure.

The above notions can be generalized to the case of the so-called $s$-concave measures. Here we assume that $s > 0$, whereas in general the notion of $s$-concave measures makes sense for any $s\in[-\infty,\infty]$. We say that a measure $\mu$ on $\mb{R}^n$ is $s$-concave if for any non-empty compact sets $A,B \subset \mb{R}^n$ we have
\begin{equation}\label{s-concave}
\mu(\lambda A+(1-\lambda)B)^s \geq \lambda \mu(A)^s +(1-\lambda) \mu(B)^s.
\end{equation}
Similarly, a measure $\mu$ is called log-concave (or $0$-concave) if for any compact sets $A,B \subset \mb{R}^n$ we have
\begin{equation}\label{log-concave}
\mu(\lambda A+(1-\lambda)B) \geq  \mu(A)^\lambda  \mu(B)^{1-\lambda}.
\end{equation}
We say that the support of measure  $\mu$ is non-degenerate  if it is not contained in any affine subspace of $\mb{R}^n$ of dimension less than $n$. It was proved by Borell (see  \cite{B}) that a measure $\mu$, with non-degenerate support, is log-concave if and only if it has a log-concave density, i.e. a density  of the form $\vp=e^{-V}$, where $V$ is convex (and may attain value $+\infty$).  Moreover, $\mu$ is $s$-concave with $s \in (0,1/n)$ if and only if it has a density $\vp$  such that $\vp^{\frac{s}{1-sn}}$ is concave. In the case $s=1/n$ the density has to satisfy the strongest condition $\vp(\lambda x+(1-\lambda)y) \geq \max(\vp(x),\vp(y))$. An example of such measure is the uniform measure on a convex body $K \subset \mb{R}^n$. Let us also notice that a measure with non-degenerate support cannot be $s$-concave with $s>1/n$. It can be seen by taking $\tilde{A}=\e A$ and $\tilde{B}=\e B$ in \eqref{s-concave}, sending $\e \to 0^+$ and comparing the limit with the Lebesgue measure.

Inequality \eqref{wBM} says that the Lebesgue measure is log-concave, whereas \eqref{clasBM} means that it is also $1/n$-concave. In general log-concavity does not imply $s$-concavity for $s>0$.  Indeed, consider the standard Gaussian measure $\gamma_n$ on $\mb{R}^n$, i.e., the measure with density $(2\pi)^{-n/2}\exp(-|x|^2/2)$. This density is clearly log-concave and therefore $\gamma_n$ satisfies \eqref{log-concave}. To see that $\gamma_n$ does not satisfy \eqref{s-concave} for $s>0$ it suffices to take $B=\{x\}$ and send $x \to \infty$. Then the left hand side converges to $0$ while the right hand side stays equal to $\lambda \mu(A)^s$, which is strictly positive for $\lambda > 0$ and $\mu(A)>0$.

One might therefore ask whether \eqref{s-concave} holds true for $\gamma_n$ if we restrict ourselves  to some special class of subsets of $\mb{R}^n$. In \cite{GZ} R. Gardner and the fourth named author conjectured (Question 7.1) that
\begin{equation}\label{GBM0}
	\gamma_n(\lambda A + (1-\lambda)B)^{1/n} \geq \lambda \gamma_n(A)^{1/n} + (1-\lambda) \gamma_n(B)^{1/n}
\end{equation}
holds true for any closed convex sets with $0 \in A \cap B$ and $\lambda \in [0,1]$ and verified this conjecture in the following cases:
\begin{itemize}
\item[(a)] when $A$ and $B$ are products of intervals containing the origin,
\item[(b)] when $A=[-a_1,a_2] \times \mb{R}^{n-1}$, where $a_1,a_2>0$ and $B$ is arbitrary,
\item[(c)] when $A = aK$ and $B = bK$ where $a, b > 0$ and $K$ is a convex set, symmetric with respect to the origin.
\end{itemize}
It is interesting to note that the  case (c) is related to the B-conjecture for Gaussian measures proposed by Banaszczyk (see \cite{L2}) and solved by Cordero-Erausquin,  Fradelizi, and Maurey (see \cite{CFM}). It states that for any convex symmetric set $K$ the function $t \mapsto \gamma_n(e^t K)$ is log-concave. The B-conjecture is asking the same question for the general class of the even  log-concave measures. It was shown in \cite{CFM} that the conjecture  is true for the case of  unconditional log-concave measures and unconditional sets (see the definition below). Moreover, the conjecture has an affirmative answer for $n = 2$ due to the works of Livne Bar-on \cite{Li} and of Saroglou \cite{S2}. In \cite{S2} the proof is done by linking the problem to the new log-Brunn-Minkowski inequality of B\"or\"oczky,  Lutwak,  Yang and  Zhang, see \cite{BLYZ1}, \cite{BLYZ2}, \cite{S1} and \cite{S2}.  In \cite{M2} the second named author proved that the assertion of the $B$-conjecture for a measure $\mu$ with a radially decreasing density and a symmetric convex body $K$ formally implies the $1/n$-concavity of the measure $\mu$ on the set of dilates of $K$.

In \cite{NT} T. Tkocz and the third named author showed that in general \eqref{GBM0} is false under the assumption $0 \in A \cap B$. For sufficiently small $\e>0$ and $\alpha < \pi/2$ sufficiently close to $\pi/2$ the pair of sets
\[
	A=\{(x,y) \in \mb{R}^2: \ y \geq |x| \tan \alpha\}, \qquad B=\{(x,y) \in \mb{R}^2: \ y \geq |x| \tan \alpha-\e\}
\]
serves as a counterexample. The authors however conjectured that \eqref{GBM0} should be true for (centrally) symmetric convex bodies $A,B$.

One of the most important Brunn-Minkowski type inequalities for the Gaussian measure is Ehrhard's inequality, which states that for any two non-empty compact sets $A,B \subset \mb{R}^n$ and any $\lambda \in [0,1]$ we have
\begin{equation}\label{Ehr}
	\Phi^{-1}(\gamma_n(\lambda A + (1-\lambda)B)) \geq \lambda \Phi^{-1}(\gamma_n(A)) + (1-\lambda)\Phi^{-1}(\gamma_n(B)),
\end{equation}
where $\Phi(t)=\gamma_1((-\infty,t])$.
This inequality has been considered for the first time by Ehrhard in \cite{E}, where the author proved it assuming that both $A$ and $B$ are convex. Then Lata{\l}a in \cite{L1} generalized Ehrhard's result to the case of arbitrary $A$ and convex $B$. In its full  generality, the inequality \eqref{Ehr} has been established by Borell, \cite{B3} (see also \cite{Ba}). Note that \eqref{GBM0}  is an inequality of the same type, with $\Phi(t)$ replaced with $t^n$, but none of them is a direct consequence of the other.  The crucial property of  Ehrhard's inequality is that it (in fact a more general form where $\lambda$ and $1-\lambda$ are replaced with $\alpha$ and $\beta$, under the conditions $\alpha+\beta \geq 1$ and $|\alpha-\beta|\leq 1$) gives the Gaussian isoperimetry as a simple consequence.

In this paper, $\mc{K}$ denotes a family of sets closed under dilations, i.e., $A \in \mc{K}$ implies $tA \in \mc{K}$ for any $t \ge 0$. In particular, we assume that for any $A \in \mc{K}$ we have $0 \in A$. Classical families of such sets include the class of star-shaped bodies, the class of convex bodies containing the origin, the class of symmetric bodies and the class of unconditional bodies.

A general form of the Brunn-Minkowski inequality can be stated as follows.

\begin{defi}
We say that a Borel measure $\mu$ on $\mb{R}^n$ satisfies the Brunn-Minkowski inequality in the class of sets $\mc{K}$ if for any $A,B \in \mc{K}$ and for any $\lambda \in [0,1]$ we have
\begin{equation}\label{BM}
	\mu(\lambda A + (1-\lambda)B)^{1/n} \geq \lambda \mu(A)^{1/n} +   (1-\lambda) \mu(B)^{1/n}.
\end{equation}
\end{defi}

\noindent Before we state our results, we introduce some basic notation and definitions.
\begin{defi}\label{df}  $\frac{}{}$

\begin{enumerate}
\item We say that a function $f:\mb{R}^n \to \mb{R}$ is unconditional if for any choice of signs $\e_1,\ldots,\e_n \in \{-1,1\}$ and any $x =(x_1,\ldots,x_n)\in \mb{R}^n$ we have $f(\e_1x_1,\ldots,\e_n x_n)=f(x)$.
\item We say that an unconditional function is decreasing if for any $1 \leq i \leq n$ and any real numbers $x_1,\ldots,x_{i-1},x_{i+1},\ldots,x_n$ the function $$t \mapsto f(x_1,\ldots,x_{i-1},t,x_{i+1},\ldots,x_n)$$ is non-increasing on $[0,\infty)$.
\item A set $A \subseteq \mb{R}^n$ is called an \emph{ideal} if $\1_A$ is unconditional and decreasing. In other words, a set $A\subset \mb{R}^n$ is an ideal if $(x_1,\ldots,x_n)\in A$ implies $(\delta_1x_1,\ldots,\delta_nx_n) \in A$ for any choice of  $\delta_1,\ldots,\delta_n \in [-1,1]$.  The class of all ideals (in $\mb{R}^n$) will be denoted by $\mc{K}_I$.
\item A set $A \subseteq \mb{R}^n$ is called symmetric if $A=-A$. The class of all symmetric convex sets in $\mb{R}^n$ will be denoted by $\mc{K}_S$.
\item A measure $\mu$ on $\mb{R}^n$ is called unconditional  if it has an unconditional  density.
\end{enumerate}
\end{defi}

We note that the class of ideals contains the class of unconditional convex bodies, but it also contains some non-convex sets. For example,  $B_p^n=\{ x\in \mb{R}^n: \sum|x_i|^p \le 1\}$ for $p\in (0,1)$ are ideals.  We also note that if an unconditional  measure $\mu$ on $\mb{R}^n$ is a product measure, i.e. $\mu=\mu_1 \otimes \ldots \otimes \mu_n$, then the measures $\mu_i$ are even  on $\mb{R}$.

Our first theorem reads as follows.

\begin{thm}\label{BMgen}
Let $\mu$ be an unconditional product measure with decreasing density. Then $\mu$ satisfies the Brunn-Minkowski inequality in the class $\mc{K}_I$ of all ideals in $\mb{R}^n$.
\end{thm}

In addition, the Examples \ref{ex1} and  \ref{ex2} at the end of the paper show that  neither the assumption that $\mu$ is a product measure, nor the unconditionality of our sets $A$ and $B$ can be dropped.

In the second part of this article we provide a link between the Brunn-Minkowski inequality and the log-Brunn-Minkowski inequality.  To state our observation we need two definitions.

\begin{defi}
Let $\mc{K}$ be a class of subsets closed under dilations. We say that a family $\odot=(\odot_\lambda)_{\lambda \in [0,1]}$ of functions $\mc{K} \times \mc{K} \to \mc{K}$ is a geometric mean if for any $A,B \in \mc{K}$ the set $A \odot_\lambda B$ is measurable, satisfies an inclusion $A \odot_\lambda B \subseteq \lambda A + (1-\lambda)B$, and   $(sA)\odot_\lambda(tB)=s^\lambda t^{1-\lambda} (A \odot_\lambda B)$, for any $s,t > 0$.
\end{defi}

\begin{defi}
We say that a Borel measure $\mu$ on $\mb{R}^n$ satisfies the log-Brunn-Minkowski inequality in the class of sets $\mc{K}$ with a geometric mean $\odot$, if for any  sets $A,B \in \mc{K}$ and for any $\lambda \in [0,1]$ we have
\[
	\mu( A \odot_\lambda B  ) \geq  \mu(A)^{\lambda}   \mu(B)^{1-\lambda}.
\]
\end{defi}

\begin{rem}\label{r1}
We shall use two different geometric means. The first one is the geometric mean $\odot^S: \mc{K}_S \times \mc{K}_S \to \mc{K}_S$, defined by the formula

\[
	A \odot_\lambda^{S} B   = \{x \in \mb{R}^n: \ \scal{x}{u} \leq h_A^{\lambda}(u) h_{B}^{1-\lambda}(u), \ \forall u \in S^{n-1} \}.
\]
Here $h_A$ is the support function of $A$, i.e., $h_A(u)=\sup_{x \in A} \scal{x}{u}$ (see, \cite{G2}, \cite{Sn}).

The second mean $\odot^I: \mc{K}_I \times \mc{K}_I \to \mc{K}_I$ is defined by
\[
	A \odot_\lambda^I B = \bigcup_{x \in A, y \in B} [-|x_1|^\lambda |y_1|^{1-\lambda}, |x_1|^\lambda |y_1|^{1-\lambda}] \times \ldots \times [-|x_n|^\lambda |y_n|^{1-\lambda}, |x_n|^\lambda |y_n|^{1-\lambda}].
\]

It is straightforward to check, with the help of the inequality $a^\lambda b^{1-\lambda} \leq \lambda a +(1-\lambda)b$, $a,b \geq 0$, that both means are indeed geometric.
\end{rem}

\noindent In the Section \ref{logBM->BM} we prove the following proposition.

\begin{prop}\label{prop}
Suppose that a Borel measure $\mu$ with a radially decreasing density $f$, i.e. density satisfying $f(tx) \geq f(x)$ for any $x \in \mb{R}^n$ and $t \in [0,1]$, satisfies the log-Brunn-Minkowski inequality, with a geometric mean $\odot$, in a certain class of sets $\mc{K}$. Then $\mu$ satisfies the Brunn-Minkowski inequality in the class $\mc{K}$.
\end{prop}

B\"or\"oczky,  Lutwak,  Yang and  Zhang \cite{BLYZ1}, proved the log-Brunn-Minkowski inequality for the Lebesgue measure and symmetric convex bodies  on  $\mb{R}^2$ equipped  with geometric mean $\odot^S$.  Saroglou \cite{S2}, generalized the inequality to the case of  measures with even log-concave densities  on  $\mb{R}^2$  (see Corollary 3.3 therein). Thus, as a consequence of Proposition \ref{prop} and Remark \ref{r1}, we get the following theorem.

\begin{thm}\label{symBM}
Let $\mu$ be a measure on $\mb{R}^2$ with an even log-concave density. Then $\mu$ satisfies the Brunn-Minkowski inequality in the class $\mc{K}_S$ of all symmetric convex sets in $\mb{R}^2$.
\end{thm}

\noindent Moreover, in \cite{CFM} (Proposition 8, see also Proposition 4.2 in \cite{S1}) the authors proved the following fact.

\begin{thm}\label{uncondlogBM}
The log-Brunn-Minkowski inequality holds true with the geometric mean $\odot^I$ for any measure with unconditional log-concave density in the class $\mc{K}_I$ of all ideals in $\mb{R}^n$.
\end{thm}

\noindent For the sake of completeness, we recall the argument in Section \ref{logBM->BM}. As a consequence, applying our Proposition \ref{prop} together with Remark \ref{r1}, we deduce:

\begin{thm}\label{uncondBM}
Let $\mu$ be an unconditional log-concave measure on $\mb{R}^n$. Then $\mu$ satisfies the Brunn-Minkowski inequality in the class $\mc{K}_I$ of all ideals in $\mb{R}^n$.
\end{thm}

The rest of this article is organized as follows. In the next section we present the proof of Theorem \ref{BMgen}.  In Section \ref{logBM->BM} we prove Proposition \ref{prop} and recall the proof of Theorem \ref{uncondlogBM}. In Section \ref{apl} we present applications of the above results. In the last section we discuss equality cases in Theorem \ref{symBM} and Theorem \ref{uncondBM}. We  also give examples showing optimality of Theorem \ref{main} and state some open questions.

\section{Proof of Theorem \ref{BMgen}}

Our strategy is to prove a certain functional version of \eqref{BM}. A functional version of the classical Brunn-Minkowski inequality is called the Pr\'ekopa-Leindler inequality, see \cite{G1} for the proof.

\vspace{0.3cm}
\noindent{\bf Pr\'ekopa-Leindler inequality,} \cite{P}, \cite{Le}:
Let $f,g,m$ be non-negative measurable functions on $\mb{R}^n$ and let $\lambda \in [0,1]$. If for all $x,y \in \mb{R}^n$ we have $m(\lambda x + (1-\lambda )y) \geq f(x)^\lambda g(y)^{1-\lambda}$ then
\[
	\call{}{}{m}{x} \geq \left( \call{}{}{f}{x} \right)^\lambda \left( \call{}{}{g}{x} \right)^{1-\lambda}.
\]

\noindent Here we prove a version of the above inequality under the assumption of unconditionality of functions $f,g$ and $m$.

\begin{prop}\label{PL}
Fix $\lambda,p \in (0,1)$. Suppose that $m,f,g$ are unconditional decreasing non-negative functions and let $\mu$ be an unconditional product measure with decreasing density on $\mb{R}^n$. Assume that for any $x,y \in \mb{R}^n$ we have
\[
	m(\lambda x + (1-\lambda)y) \geq f(x)^p g(y)^{1-p}.
\]
Then
\[
	\call{}{}{m}{\mu} \geq  \left[\left( \frac{\lambda}{p} \right)^{p}  \left( \frac{1-\lambda}{1-p} \right)^{1-p} \right]^n \left( \call{}{}{f}{\mu} \right)^p \left( \call{}{}{g}{\mu} \right)^{1-p}.
\]
\end{prop}
\noindent The above proposition  allows us to prove the following lemma, which is in fact a reformulation of Theorem \ref{BMgen}.

\begin{lem}\label{main}
Let $A,B$ be ideals in $\mb{R}^n$ and let $\mu$ be an unconditional product measure with decreasing density on $\mb{R}^n$. Then for any $\lambda \in [0,1]$ and $p \in (0,1)$ we have
\[
	\mu(\lambda A + (1-\lambda) B) \geq \left[\left( \frac{\lambda}{p} \right)^{p}  \left( \frac{1-\lambda}{1-p} \right)^{1-p} \right]^n \mu(A)^{p} \mu(B)^{1-p}.
\]
\end{lem}
It is worth noticing that the factor on the right hand side of this inequality replaces in some sense the lack of homogeneity of our measure $\mu$. The main idea of the proof is to introduce an additional parameter $p \ne \lambda$ and do the optimization with respect to $p$.

We first show how Lemma \ref{main} implies Theorem \ref{BMgen}.

\begin{proof}[Proof of Theorem \ref{BMgen}]
Without loss of generality we assume that $\lambda \in (0,1)$. Let us assume for a moment that $\mu(A)\mu(B)>0$. Then we can use Lemma \ref{main} with
\begin{equation}\label{p}
	p = \frac{\lambda \mu(A)^{1/n}}{\lambda \mu(A)^{1/n} + (1-\lambda) \mu(B)^{1/n}} \in (0,1).
\end{equation}
Note that
\[
	 \frac{\lambda}{p} =  \frac{\lambda \mu(A)^{1/n} + (1-\lambda) \mu(B)^{1/n}}{ \mu(A)^{1/n}}, \qquad  \frac{1-\lambda}{1-p} =  \frac{\lambda \mu(A)^{1/n} + (1-\lambda) \mu(B)^{1/n}}{ \mu(B)^{1/n}}.
\]
Then
\begin{align*}
\left[\left( \frac{\lambda}{p} \right)^{p}  \left( \frac{1-\lambda}{1-p} \right)^{1-p} \right]^n \mu(A)^{p} \mu(B)^{1-p} = \left( \lambda \mu(A)^{1/n} + (1-\lambda) \mu(B)^{1/n} \right)^n.
\end{align*}
Thus the inequality in Lemma \ref{main} becomes
\[
	\mu(\lambda A + (1-\lambda) B) \geq \left( \lambda \mu(A)^{1/n} + (1-\lambda) \mu(B)^{1/n} \right)^n.
\]

Now suppose that, say, $\mu(B)=0$. Since $B$ is a non-empty ideal, we have $0 \in B$. Therefore, $\lambda A \subseteq \lambda A + (1-\lambda) B$. Let $\vp$ be the unconditional decreasing density of $\mu$. Hence,
\begin{align*}
	\mu(\lambda A + (1-\lambda) B) & \geq \mu(\lambda A) = \call{\lambda A}{}{\vp(x)}{x} = \lambda ^n \call{A}{}{\vp(\lambda y)}{y} \\
	& = \lambda ^n \call{A}{}{\vp(\lambda y_1, \ldots, \lambda y_n)}{y} = \lambda ^n \call{A}{}{\vp(\lambda |y_1|, \ldots, \lambda |y_n|)}{y} \\
	& \geq  \lambda ^n \call{A}{}{\vp( |y_1|, \ldots,  |y_n|)}{y} = \lambda^n \mu(A).
\end{align*}
Therefore,
\[
	\mu(\lambda A + (1-\lambda) B)^{1/n} \geq \lambda \mu(A)^{1/n} = \lambda \mu(A)^{1/n} + (1-\lambda) \mu(B)^{1/n}.
\]
\end{proof}

Next we show that Proposition \ref{PL} implies Lemma \ref{main}.

\begin{proof}[Proof of Lemma \ref{main}]
We can assume that $\lambda \in (0,1)$. Let us take $m(x)=\1_{\lambda A + (1-\lambda)B}(x)$, $f(x)=\1_A(x)$, $g(x)=\1_B(x)$. Clearly, $f,g$ and $m$ are unconditional and decreasing, and verify $m(\lambda x + (1-\lambda )y) \geq f(x)^pg(y)^{1-p}$ for any $p \in (0,1)$. Our assertion follows from Proposition \ref{PL}.
\end{proof}

\noindent For the proof of Proposition \ref{PL} we need a one dimensional Brunn-Minkowski inequality for unconditional measures.

\begin{lem}\label{dim1}
Let $A,B$ be two symmetric intervals and let $\mu$ be an unconditional measure with decreasing density on $\mb{R}$. Then for any $\lambda \in [0,1]$ we have
\[
	\mu(\lambda A + (1-\lambda)B) \geq \lambda \mu(A)+(1-\lambda) \mu(B).
\]	
\end{lem}
\begin{proof}
We can assume that $A=[-a,a]$ and $B=[-b,b]$ for some $a,b>0$. Let $\vp$ be the density of $\mu$. Then our assertion is equivalent to
\[
	\call{0}{\lambda a + (1-\lambda) b}{\vp(x)}{x} \geq  \lambda \call{0}{ a}{\vp(x)}{x} + (1-\lambda)\call{0}{ b}{\vp(x)}{x}.
\]
In other words, the function $t \mapsto \call{0}{t}{\vp(x)}{x}$ should be concave on $[0,\infty)$. This is equivalent to $t \mapsto \vp(t)$ being non-increasing on $[0,\infty)$.
\end{proof}

\begin{proof}[Proof of Proposition \ref{PL}]
We proceed by induction on $n$. Let us begin with the case $n=1$. We can assume that $\norma{f}{\infty},\norma{g}{\infty}>0$. If we multiply the functions $m,f,g$ by positive numbers $c_m, c_f, c_g$ satisfying $c_m=c_f^p c_g^{1-p}$, the hypothesis and the assertion do not change. Therefore, taking $c_f= \norma{f}{\infty}^{-1}$, $c_g= \norma{g}{\infty}^{-1}$, $c_m= \norma{f}{\infty}^{-p} \norma{g}{\infty}^{-(1-p)}$ we can assume that $\norma{f}{\infty}=\norma{g}{\infty}=1$. Then the sets $\{f > t\}$ and $\{g > t\}$ are non-empty for $t \in (0,1)$. Moreover, $\lambda\{f> t\}+(1-\lambda)\{g > t\} \subseteq \{m > t\}$. Indeed, if $x \in \{f > t\}$ and $y \in \{g > t\}$ then $m(\lambda x+(1-\lambda)y) \geq f(x)^pg(y)^{1-p} > t^pt^{1-p}=t$. Thus, $\lambda x+(1-\lambda)y \in \{m > t\}$. Therefore, using Lemma \ref{dim1}, we get
\begin{align*}
 \call{}{}{m}{\mu} = \call{0}{\infty}{\mu(\{m>t\})}{t} \geq & \call{0}{1}{\mu(\lambda\{f > t\}+(1-\lambda)\{g > t\})}{t} \\ & \geq \lambda \call{0}{1}{\mu(\{f > t\})}{t} + (1-\lambda) \call{0}{1}{\mu(\{g > t\})}{t} \\
  & = \lambda \call{}{}{f}{\mu} + (1-\lambda) \call{}{}{g}{\mu}. \end{align*}
Now, using the inequality $pa+(1-p)b \geq a^pb^{1-p}$, $a,b \geq 0$, we get
\begin{align}\label{p=l}
\lambda \call{}{}{f}{\mu} + (1-\lambda) \call{}{}{g}{\mu}
   & = p\frac{\lambda}{p} \call{}{}{f}{\mu} + (1-p)\frac{1-\lambda}{1-p} \call{}{}{g}{\mu} \\
   & \geq  \left( \frac{\lambda}{p} \right)^{p}  \left( \frac{1-\lambda}{1-p} \right)^{1-p} \left( \call{}{}{f}{\mu} \right)^p \left( \call{}{}{g}{\mu} \right)^{1-p}.
\end{align}

Next, we do the induction step. Let us assume that the assertion is true in dimension $n-1$. Let $m,f,g:\mb{R}^{n} \to [0,\infty)$ be unconditional decreasing. For $x_0,y_0,z_0 \in \mb{R}$ we define functions $m_{z_0},f_{x_0}, g_{y_0}$ by
\[
m_{z_0}(x) = m(z_0, x), \quad f_{x_0}(x)=f(x_0,x), \quad g_{y_0}(x)=g(y_0,x).
\]
Clearly, these functions are also unconditional. Moreover, due to our assumptions on $m,f,g$ we have
\begin{align*}
	m_{\lambda x_0+(1-\lambda)y_0}(\lambda x + (1-\lambda)y) &= m(\lambda x_0+(1-\lambda)y_0, \lambda x + (1-\lambda)y) \\
	& \geq f(x_0,x)^p g(y_0,y)^{1-p} = f_{x_0}(x)^p g_{y_0}(y)^{1-p}.
\end{align*}
Let us decompose $\mu$ in the form $\mu=\mu_1 \times \bar{\mu}$, where $\mu_1$ is a measure on $\mb{R}$. Note that $\mu_1$ and $\bar{\mu}$ are unconditional and $\bar{\mu}$  is a product measure on $\mb{R}^{n-1}$.
Thus, by our induction assumption we have
\begin{equation}\label{induction}
	\call{}{}{m_{\lambda x_0 + (1-\lambda)y_0}}{\bar{\mu}} \geq  \left[\left( \frac{\lambda}{p} \right)^{p}  \left( \frac{1-\lambda}{1-p} \right)^{1-p} \right]^{n-1} \left( \call{}{}{f_{x_0}}{\bar{\mu}} \right)^p \left( \call{}{}{g_{y_0}}{\bar{\mu}} \right)^{1-p}.
\end{equation}
Now we define the functions
\begin{equation}\label{M}
	M(z_0)= \left[\left( \frac{\lambda}{p} \right)^{p}  \left( \frac{1-\lambda}{1-p} \right)^{1-p} \right]^{-(n-1)} \call{}{}{m_{z_0}(\xi)}{\bar{\mu}(\xi)},
\end{equation}	
\begin{equation}	\label{FG}
 F(x_0)= \call{}{}{f_{x_0}(\xi)}{\bar{\mu}(\xi)}, \qquad \qquad G(y_0)= \call{}{}{g_{y_0}(\xi)}{\bar{\mu}(\xi)}.
\end{equation}
Using inequality (\ref{induction}) we immediately get that
\[
M(\lambda x_0+(1-\lambda) y_0) \geq F(x_0)^p G(y_0)^{1-p}.
\]
Moreover, it is easy to see that $M,F,G$ are unconditional decreasing on $\mb{R}$. Thus,  using Lemma \ref{dim1} (the one-dimensional case), we get
\begin{equation}\label{1dimFGM}
	\call{}{}{M(z_0)}{\mu_1(z_0)} \geq \left( \frac{\lambda}{p} \right)^{p}  \left( \frac{1-\lambda}{1-p} \right)^{1-p} \left( \call{}{}{F(x_0)}{\mu_1(x_0)} \right)^p \left( \call{}{}{G(y_0)}{\mu_1(y_0)}\right)^{1-p}.
\end{equation}
Observe that
\begin{align*}
	\call{}{}{M(z_0)}{\mu_1(z_0)} & = \left[\left( \frac{\lambda}{p} \right)^{p}  \left( \frac{1-\lambda}{1-p} \right)^{1-p} \right]^{-(n-1)} \call{}{}{\call{}{}{m_{z_0}(\xi)}{\mu_{n-1}(\xi)}}{\mu_1(z_0)} \\
	& = \left[\left( \frac{\lambda}{p} \right)^{p}  \left( \frac{1-\lambda}{1-p} \right)^{1-p} \right]^{-(n-1)} \call{}{}{m}{\mu}.
\end{align*}
Similarly,
\[
	\call{}{}{F(x_0)}{\mu_1(x_0)}= \call{}{}{f}{\mu}, \qquad \call{}{}{G(y_0)}{\mu_1(y_0)}= \call{}{}{g}{\mu}.
\]
Our assertion follows.

\end{proof}

\section{Proof of Proposition \ref{prop}}  \label{logBM->BM}

In this section we first prove Proposition \ref{prop}. The argument has a flavour of our previous proof.

\begin{proof}[Proof of Proposition \ref{prop}]

Let us first assume that $\mu(A)\mu(B)>0$. From the definition of geometric mean we have $A \odot_p B \subseteq p A + (1-p)B$, for any $p \in (0,1)$. Thus,
\begin{align*}
	\mu(\lambda A+(1-\lambda)B) & = \mu\left( p \cdot   \frac{\lambda}{p} A+(1-p) \cdot \frac{1-\lambda}{1-p} B \right) \geq \mu\left(\left( \frac{\lambda}{p} A \right) \odot_p \left( \frac{1-\lambda}{1-p} B \right) \right) \\
	& = \mu\left(\left( \frac{\lambda}{p}  \right)^p  \left( \frac{1-\lambda}{1-p}  \right)^{1-p}  A \odot_p B \right).
\end{align*}
Let $t =\left( \frac{\lambda}{p}  \right)^p \left( \frac{1-\lambda}{1-p}  \right)^{1-p}$ and $C= A \odot_p B$. From the concavity of the logarithm it follows that $0 \leq t \leq 1$. We have
\begin{equation}\label{homogenity}
	\mu(tC) = \call{tC}{}{f(x)}{x} = t^n \call{C}{}{f(tx)}{x} \geq t^n \call{C}{}{f(x)}{x} = t^n \mu(C).
\end{equation}
Therefore,
\[
	\mu(\lambda A+(1-\lambda)B) \geq  t^n \mu( A \odot_p B) \geq t^n \mu(A)^p \mu(B)^{1-p} = \left[ \left( \frac{\lambda}{p}  \right)^p \left( \frac{1-\lambda}{1-p}  \right)^{1-p} \right]^n \mu(A)^p \mu(B)^{1-p}.
\]
Taking
\begin{equation}\label{optp}
	p = \frac{\lambda \mu(A)^{1/n}}{\lambda \mu(A)^{1/n} + (1-\lambda) \mu(B)^{1/n}}
\end{equation}
gives
\[
	\mu(\lambda A + (1-\lambda)B)^{1/n} \geq \lambda \mu(A)^{1/n} + (1-\lambda)\mu(B)^{1/n}.
\]
If, say, $\mu(B)=0$ then by \eqref{homogenity}, applied for $C$ replaced with $A$, and the fact that $0 \in B$ we get
\[
	\mu(\lambda A+(1-\lambda)B)^{1/n} \geq \mu(\lambda A)^{1/n} \geq \lambda \mu(A)^{1/n} =  \lambda \mu(A)^{1/n} + (1-\lambda)\mu(B)^{1/n}.
\]
\end{proof}

We now sketch the proof of Theorem \ref{uncondlogBM}.

\begin{proof}
Let $A,B \in \mc{K}_I$ and let us take $f,g,m:[0,+\infty)^n \to [0,+\infty)$ given by $f=\1_{A \cap [0,+\infty)^n}$, $g=\1_{B \cap [0,+\infty)^n}$ and $m=\1_{(A \odot_\lambda^I B) \cap [0,+\infty)^n}$. Let $\vp$ be the unconditional log-concave density of $\mu$. We define
\[
	F(x)=f(e^{x_1},\ldots, e^{x_n}) \vp(e^{x_1},\ldots, e^{x_n}) e^{x_1+\cdots+x_n}, \quad G(x)=g(e^{x_1},\ldots, e^{x_n})\vp(e^{x_1},\ldots, e^{x_n}) e^{x_1+\cdots+x_n},
\]
\[
M(x)=m(e^{x_1},\ldots, e^{x_n})\vp(e^{x_1},\ldots, e^{x_n}) e^{x_1+\cdots+x_n} .
\]
One can easily check, using the definition of $\mc{K}_I$ and the definition of the geometric mean $\odot_\lambda^I$, as well as the inequalities
\begin{align*}
	& \vp(e^{\lambda x_1 + (1-\lambda)y_1},\ldots, e^{\lambda x_n + (1-\lambda)y_n}) \\
	& \qquad \geq \vp(\lambda e^{x_1} + (1-\lambda)e^{y_1},\ldots, \lambda e^{x_n} + (1-\lambda)e^{y_n}) \geq \vp(e^{x_1},\ldots,e^{x_n})^\lambda \vp(e^{y_1},\ldots,e^{y_n})^{1-\lambda},
\end{align*}
that the functions $F,G,M$ satisfy the assumptions of the Pr\'ekopa-Leindler inequality. As a consequence, we get $\mu((A \odot_\lambda^I B) \cap [0,+\infty)^n) \geq \mu(A \cap [0,+\infty)^n)^\lambda \mu( B \cap [0,+\infty)^n)^{1-\lambda}$. The assertion follows from unconditionality of our measure $\mu$ and the fact that $A,B$ and $A \odot_\lambda^I B$ are ideals.
\end{proof}

\section{Applications} \label{apl}

Let us describe some corollaries of the Brunn-Minkowski type inequality we established, which are analogues to well-known offsprings of the Brunn-Minkowski inequality for the volume. In what follows a pair $(\mc{K},\mu)$ is called \emph{nice} if one of the following three cases holds.
\begin{itemize}
\item[(a)] $\mc{K}=\mc{K}_I$ and $\mu$ is an unconditional, product measure with decreasing density on $\mb{R}^n$,
\item[(b)] $\mc{K}=\mc{K}_I$ and $\mu$ is an unconditional log-concave measure on $\mb{R}^n$,
\item[(c)] $\mc{K}=\mc{K}_S$ and $\mu$ is an even log-concave measure on $\mb{R}^2$.
\end{itemize}

\begin{cor}\label{parvol}
Suppose that a pair $(\mc{K},\mu)$ is nice. Let $A,B \subset \mc{K}$ be convex. Then the function $t \mapsto \mu(A+tB)^{1/n}$ is concave on $[0,\infty)$.
\end{cor}

\noindent Indeed, for any $\lambda \in [0,1]$ and $t_1,t_2 \geq 0$ we have
\begin{align*}
	\mu(A+(\lambda t_1 + (1-\lambda)t_2)B)^{1/n} & = \mu(\lambda (A + t_1 B)+(1-\lambda) (A + t_2 B))^{1/n} \\
	& \geq \lambda  \mu(A + t_1 B)^{1/n} + (1-\lambda) \mu(A + t_2 B)^{1/n}.
\end{align*}
Note that in the first line we have used the convexity of $A$ and $B$. If $B=B_2^n$ is the unit Euclidean ball, the expression $\mu(A+tB)$ is called the parallel volume and has been studied in the case of the Lebesgue measure by Costa and Cover in \cite{C} as an analogue of concavity of entropy power in Information theory. The authors conjectured that for any measurable set $A$ the parallel volume is $1/n$-concave. In \cite{FM}, M. Fradelizi and the second named author proved that this conjecture is true for any measurable set in dimension 1 and for any connected set in dimension 2. However, the authors proved that this conjecture fails for arbitrary sets in dimension $n \geq 2$. In a recent paper \cite{M} the second named author investigated the parallel volume $\mu(A+ tB_2^n)$ in the context of $s$-concave measures as well as functional versions. Our Corollary \ref{parvol} gives the Costa-Cover conjecture for any convex set $A \in \mc{K}$, where $(\mc{K},\mu)$ is a nice pair. Moreover, $B_2^n$ can be replaced with any convex set $B \in \mc{K}$.

Second, we  state the following analogue of Brunn's theorem on volumes of sections of convex bodies (see \cite{G1}, \cite{G2} and \cite{Sn} for the volume case).

\begin{cor}\label{Brunn}
Suppose that a pair $(\mc{K},\mu)$ is nice. Let $A \in \mc{K}$ be a convex set and let $\vp$ be the density of $\mu$. Then the function $t \mapsto \mu_{n-1}(A \cap \{x_1=t\})$ is $\frac{1}{n-1}$-concave on its support, where
\[
\mu_{n-1}(A \cap \{x_1=t\}) = \call{(t, x_2, \dots, x_n) \in A}{}{\vp(t, x_2, \dots, x_n)}{x_2 \dots \mathrm{d} x_n}.
\]
\end{cor}

\noindent Indeed, let us denote $A_{\{x_1=t\}}=A \cap \{x_1=t\}$. By convexity of $A$ we get $$\lambda A_{\{x_1=t_1\}} + (1-\lambda)A_{\{x_1=t_2\}} \subseteq  A_{\{x_1=\lambda t_1 + (1-\lambda)t_2\}}.$$ Thus, using \eqref{BM}, for any $\lambda \in [0,1]$ and $t_1, t_2 \in \mb{R}$ such that $A_{\{x_1=t_1\}}$ and $A_{\{x_1=t_2\}}$ are both non-empty, we get
\begin{eqnarray*}
	\mu_{n-1}(A_{\{x_1=\lambda t_1 + (1-\lambda)t_2\}})^{\frac{1}{n-1}} &\geq& \mu_{n-1}(\lambda A_{\{x_1=t_1\}} + (1-\lambda)A_{\{x_1=t_2\}})^{\frac{1}{n-1}} \\
	&\geq& \lambda \mu_{n-1}(A_{\{x_1=t_1\}})^{\frac{1}{n-1}} + (1-\lambda)\mu_{n-1}(A_{\{x_1=t_2\}})^{\frac{1}{n-1}}.
\end{eqnarray*}

Third, let us mention the relation of our result to the Gaussian isoperimetric inequality and the S-inequality. The  Gaussian isoperimetric inequality (established by Sudakov and Tsirelson, \cite{ST}, and independently by Borell, \cite{B2}),   states that for any measurable set $A \subset \mb{R}^n$ and any $t>0$, the quantity $\gamma_n(A_t)$ is minimized, among all sets with prescribed measure, for the half spaces $H_{a,\theta} = \{x \in \mb{R}^n: \ \scal{x}{\theta} \leq a\}$, with $a \in \mb{R}$ and $\theta \in S^{n-1}$. Infinitesimally, it says that among all sets with prescribed measure the half spaces are those with the smallest Gaussian surface area, i.e., the quantity
\[	
	\gamma_n^+(\partial A) = \liminf_{t \to 0^+} \frac{ \gamma_n(A+tB_2^n)-\gamma_n(A) }{t}.
\]

The S-inequality of Lata{\l}a and Oleszkiewicz, see \cite{LO}, states that for any $t>1$ and any symmetric convex body $A$ the quantity $\gamma_n(tA)$ is minimized, among all subsets with prescribed measure, for the strip of the form $S_L =\{x \in \mb{R}^n: \ |x_1| \leq L\}$. This result admits an equivalent infinitesimal version, namely, among all symmetric convex bodies $A$ with prescribed Gaussian measure the strip $S_L$ minimizes the quantity $\po{}{t}\gamma_n(tA)\big|_{t=1}$, which is equivalent to maximizing
\[
	M_{\gamma_n}(A) = \call{A}{}{|x|^2}{\gamma_n(x)},
\]
see \cite{KS} or \cite{NT3}.
For a general measure $\mu$ with a density $e^{-\psi}$, one can show that the infinitesimal version of S-inequality is an issue of maximizing the quantity
\begin{equation}
	M_\mu(A) = \call{A}{}{\scal{x}{\nabla \psi (x)}}{\mu(x)},
\end{equation}
see equation \eqref{infs} below.
Not much is known about an analogue of $S$-inequality in the case of general measure. In the unconditional case it has been solved for some particular product measures like products of Gamma and Weibull distributions, see \cite{NT2}. It turns out that inequality \eqref{GBM0} implies a certain mixture of Gaussian isoperimetry and reverse S-inequality. Namely, we have the following corollary.

\begin{cor}\label{MinFir}
Let $A$ be an ideal in $\mb{R}^n$ (or a general symmetric convex set in $\mb{R}^2$) and let $r>0$. Then we have
\[
	r \gamma_n^+(\partial A) + M_{\gamma_n}(A) \geq n \gamma_n(rB_2^n)^{\frac{1}{n}} \gamma_n(A)^{1-\frac{1}{n}}
\]
with equality for $A=rB_2^n$.
\end{cor}

\noindent
Let us note that
\begin{eqnarray*}
	\gamma_n(rB_2^n + \e B_2^n) & = & (2 \pi)^{-n/2}(r+\e)^n \call{B_2^n}{}{e^{-\frac{|(r+\e)x|^2}{2}}}{x} \\ & = & (2 \pi)^{-n/2}(r^n + nr^{n-1}\e + o(\e)) \call{B_2^n}{}{e^{-\frac{|rx|^2}{2}}(1-\e r|x|^2 + o(\e))}{x} \\ & = & \gamma_n(rB_2^n) + \frac{\e}{r} \left( n\gamma_n(rB_2^n) - M_{\gamma_n}(rB_2^n) \right) + o(\e).
\end{eqnarray*}
Thus,
\[
	r \gamma_n^+(\partial (rB_2^n)) = n \gamma_n(rB_2^n) - M_{\gamma_n}(rB_2^n).
\]
Hence, if $\gamma_n(A) = \gamma_n(rB_2^n)$ in Corollary \ref{MinFir}, then we get
\begin{eqnarray}\label{isop}
	r \gamma_n^+(\partial A) + M_{\gamma_n}(A) \geq r \gamma_n^+(\partial (rB_2^n)) + M_{\gamma_n}(rB_2^n).
\end{eqnarray}
In other words, Euclidean balls minimize the quantity $r \gamma_n^+(\partial A) + M_{\gamma_n}(A)$ among ideals in $\mb{R}^n$ (or symmetric convex sets in $\mb{R}^2$) with prescribed measure.

It is known that among all symmetric convex sets (in fact among all measurable sets) with prescribed Gaussian measure, the quantity $M_{\gamma_n}(A)$ is minimized by Euclidean balls $rB_2^n$ (this fact can be seen as a reverse S-inequality).  Indeed, suppose that $\gamma_n(A)=\gamma_n(rB_2^n)$. Then
\begin{align*}
	M_{\gamma_n}(A) - M_{\gamma_n}(rB_2^n) & = \call{A \setminus (rB_2^n)}{}{|x|^2}{\gamma_n(x)} - \call{(rB_2^n) \setminus A}{}{|x|^2}{\gamma_n(x)} \\
	&  \geq  r^2(\gamma_n(A \setminus (rB_2^n))-\gamma_n((rB_2^n) \setminus A)) = 0.
\end{align*}
However, in general the quantity $\gamma_n^+(\partial A)$ is not minimized by Euclidean balls, e.g., one can check that for large values of $\gamma_2(A)$ the symmetric strip has smaller Gaussian surface area than the Euclidean ball, see \cite[Lemma 3]{LO2}. Hence, inequality \eqref{isop} is a new isoperimetric-type inequality that links the Gaussian isoperimetry and reverse S-inequality.

Let us state and prove a more general version of Corollary \ref{MinFir}. Let $\mu^+(\partial A)$ be the $\mu$ surface area of $A$, i.e.,
\[
\mu^+(\partial A)=\liminf_{t \to 0^+} \frac{\mu(A+tB_2^n)-\mu(A)}{t}.
\]
Let
\[
	V_1^\mu(A,B) = \frac{1}{n} \liminf_{t \to 0^+} \frac{\mu(A+tB)-\mu(A)}{t}
\]
be the first mixed volume of arbitrary sets $A$ and $B$, with respect to measure $\mu$. Clearly, $\mu^+(\partial A)=nV_1^\mu(A,B_2^n)$.
\begin{cor}\label{MinFirGen}
Let $A,B \in \mc{K}$ and suppose that $(\mc{K},\mu)$ is a nice pair. Then we have
\begin{equation}\label{MF}
	V_1^\mu(A,B) + \frac{1}{n} M_\mu(A) \geq \mu(B)^{1/n} \mu(A)^{1-1/n}.
\end{equation}
In particular,
\begin{equation}\label{MFB}
	r \mu^+(\partial A) + M_\mu(A) \geq n \mu(r B_2^n)^{1/n} \mu(A)^{1-1/n}.
\end{equation}
\end{cor}

\noindent To prove this we note that for any sets $A,B \in \mc{K}$ and any $\e \in [0,1)$ we have
\begin{equation}\label{newisop}
	\mu(A+\e B)^{1/n} \geq (1-\e) \mu\left( \frac{A}{1-\e} \right)^{1/n} + \e \mu(B)^{1/n}.
\end{equation}
Indeed, it suffices to use Theorem \ref{BMgen} with $\lambda = 1-\e$ and $\tilde{A}=A/(1-\e)$, $\tilde{B}=B$. Note that for $\e=0$ we have equality. Thus, differentiating \eqref{newisop} at $\e=0$ we get
\[
	\frac{1}{n} \mu(A)^{\frac{1}{n}-1} \cdot n V_1^\mu(A,B) \geq \mu(B)^{\frac{1}{n}} -\mu(A)^{\frac{1}{n}} + \frac{1}{n}   \mu(A)^{\frac{1}{n}-1} \po{}{t} \mu(tA)\Big|_{t=1}.
\]
By changing variables we obtain
\begin{equation}\label{infs}
	\po{}{t} \mu(tA)\Big|_{t=1} = \po{}{t}  \call{A}{}{e^{-\psi(tx)}t^n}{}{x} \Big|_{t=1} = n \mu(A) - \call{A}{}{\scal{x}{\nabla \psi(x)}}{\mu(x)} = n \mu(A)-M_\mu(A).
\end{equation}
Thus,
\[
	\mu(A)^{\frac{1}{n}-1}  V_1^\mu(A,B) \geq \mu(B)^{\frac{1}{n}}  - \frac{1}{n}   \mu(A)^{\frac{1}{n}-1}M_\mu(A),
\]
which is exactly \eqref{MF}. To get \eqref{MFB} one has to take $B=rB_2^n$ in \eqref{MF}.

The above inequalities can be seen as an analogue of the so-called Minkowski first inequality for the Lebesgue measure (see \cite{G1}, \cite{G2} and \cite{Sn}), which says that for any two convex bodies $A,B$ in $\mb{R}^n$ we have
\[
	V_1^{\vol_n}(A,B) \geq \vol_n(A)^{1-\frac{1}{n}} \vol_n(B)^{\frac{1}{n}}.
\]

\section{Examples and open problems}\label{3}

We first discuss equality cases in Theorem \ref{symBM} and Theorem \ref{uncondBM}.

\begin{rem}
The equality in Theorem 2 and Theorem 4 is achieved only if $A$ is a dilation of $B$. Indeed, in the proof of Proposition \ref{prop} we use the inclusion $\tilde{A} \odot_p \tilde{B} \subseteq p \tilde{A} + (1-p) \tilde{B}$, where $\tilde{A}=\frac{\lambda}{p}A$ and $\tilde{B}=\frac{1-\lambda}{1-p}B$, with $p$ given by \eqref{optp}. To have equality in (7) we need to have, in particular, equality in the above inclusion (with this particular choice of $p$). Notice that $a^pb^{1-p}=pa+(1-p)b$, $a,b \geq 0$, if and only if $a=b$. Thus, $\tilde{A} \odot_p^S \tilde{B} = p \tilde{A} + (1-p) \tilde{B}$ if and only if $\tilde{A}=\tilde{B}$ (by using the fact that $h_{\tilde{A}}=h_{\tilde{B}}$ if and only if $\tilde{A}=\tilde{B}$). Similarly, one has $\tilde{A} \odot_p^I \tilde{B} = p \tilde{A} + (1-p) \tilde{B}$ if and only if $\tilde{A}=\tilde{B}$. This means that $A$ is a dilation of $B$.

In general one cannot hope to have equality cases only if $A=B$. Let us illustrate this in the case of the Lebesgue measure. Indeed, then we have equality in (7) if $A=aK$ and $B=bK$, where $K$ is some fixed convex set. In this case the equality $\tilde{A}=\tilde{B}$ leads to the condition $\frac{\lambda}{p}a=\frac{1-\lambda}{1-p}b$, which is equivalent to choosing $p=\frac{\lambda a}{\lambda a + (1-\lambda)b}$. This coincides with \eqref{optp}.

However, one can get $A=B$ as the only case of equality if one assumes that the density of $\mu$ is strictly decreasing. To see this it suffices to observe that for the equality in (7) we have to have $t=1$ in the proof of Proposition \ref{prop}, which leads to $\mu(A)=\mu(B)$. Together with the fact that $A$ is a dilation of $B$ we get $A=B$.
\end{rem}

We also show that the assumptions of Theorem 1 are necessary. Namely, as long as we work with decreasing densities, which may not be log-concave, one has to assume that the measure is product and the sets are unconditional.

\begin{ex}\label{ex1}
The assumption, that our measure $\mu$ in Theorem \ref{BMgen} is a product, is important. Indeed, let us take the square $C=\{|x|,|y| \leq 1\} \subset \mb{R}^2$ and take the measure with density $\vp(x)=\frac12 \1_{2C}(x)+\frac12 \1_{C}(x)$. This density is unconditional, however it is not a product. Let us define $\psi(a)=\sqrt{\mu(aC)}$. The assertion of Theorem \ref{BMgen} implies that $\psi$ is concave. However, we have $\psi(a)=\sqrt{2a^2+2}$ for $a \in [1,2]$, which is strictly convex. Thus, $\mu$ does not satisfy \eqref{BM}.
\end{ex}

\begin{ex}\label{ex2}
In general, under the assumption that our measure $\mu$ is unconditional and a product, one cannot prove that Theorem \ref{BMgen} holds true for arbitrary symmetric convex sets.
To see this, let us take the product measure $\mu=\mu_0 \otimes \mu_0$ on $\mb{R}^2$, where $\mu_0$ has an unconditional density $\vp(x)=p+(1-p)\1_{[-1/\sqrt{2},1/\sqrt{2}]}(x)$ for some $p \in [0,1]$.

To simplify the computation let us rotate the whole picture by angle $\pi/4$. Then consider the rectangle $R=[-1,1] \times [-\lambda , \lambda ]$ for $0 < \lambda \leq 1/2$. As in the previous example, it is enough to show that the function $\psi(a)=\sqrt{\mu(aR)}$ is not concave. Let us consider this function only on the interval $[1/ \lambda,\infty)$. The condition $\lambda \leq 1/2$ ensures that the point $(a,\lambda a)$ lies in the region with density $p^2$. Let us introduce lengths $l_1,l_2,l_3$ (see the picture below).

\begin{figure}[h!]
\centering
\def\svgwidth{230pt}
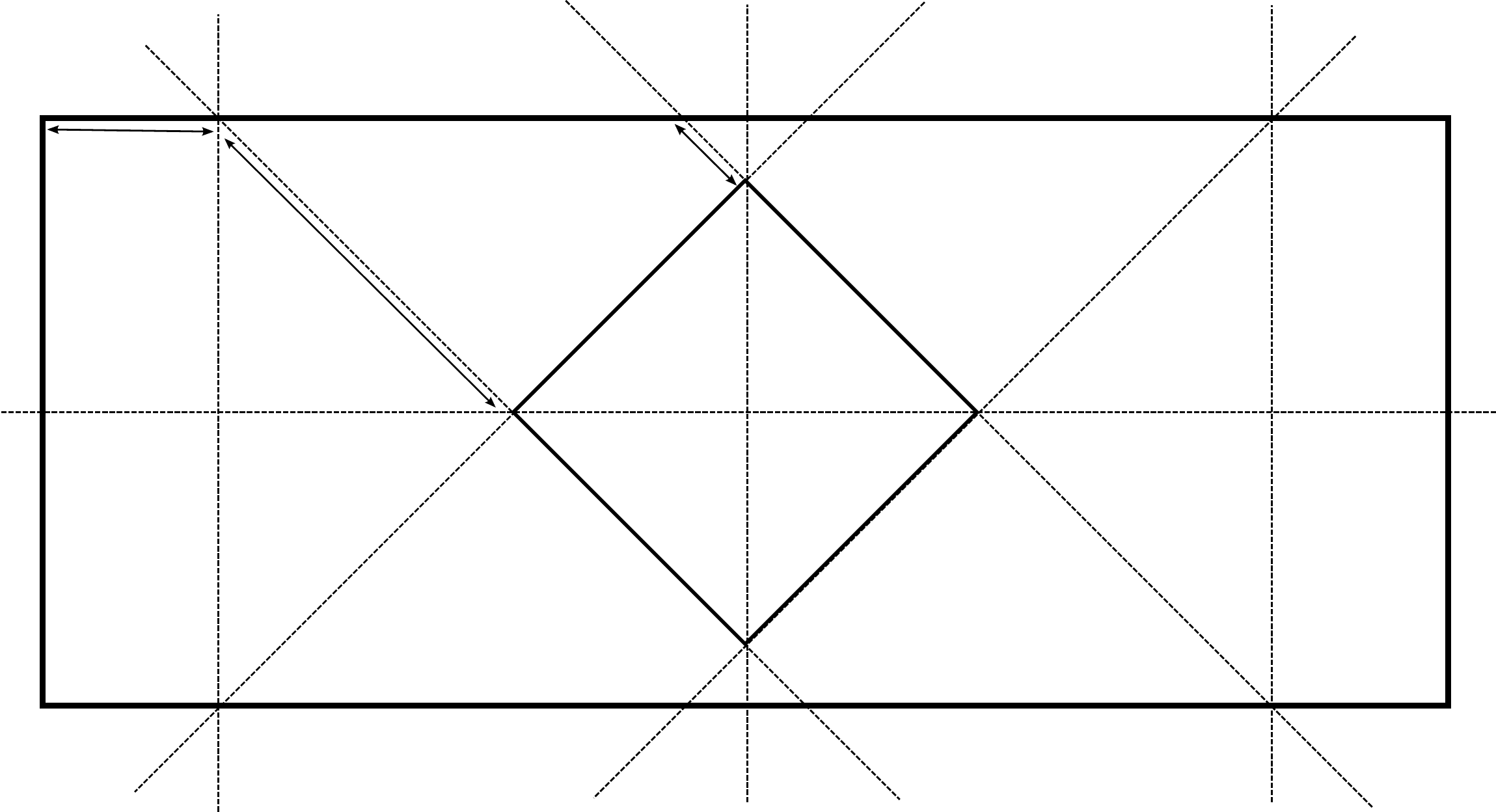
\end{figure}
\end{ex}

\noindent Note that $l_1=\sqrt{2}\lambda a$, $l_2=\sqrt{2}(\lambda a-1)$ and $l_3=a-(1+\lambda a)$. Let $\omega(a)=\mu(aR)$. We have
\begin{align*}
	\omega(a) & = 2+4\sqrt{2}p \cdot \frac{l_1+l_2}{2} + p^2l_1^2 + p^2l_2^2 + 4p^2l_3 \lambda a \\ &  = 2+4p(2\lambda a -1) + 2p^2 \lambda^2 a^2  + 2p^2(\lambda a -1)^2 + 4p^2 \lambda a(a-1-\lambda a) \\
	& = 2(1-p)^2 + 4p\lambda a(pa+2-2p)=d_0+d_1a+d_2a^2,
\end{align*}
where $d_0=2(1-p)^2$, $d_1=8p(1-p)\lambda$, $d_2=4p^2\lambda$.
We show that $\psi$ is strictly convex for $p \in (0,1)$ and $0<\lambda<1/2$.  Indeed, $\psi''>0$ is equivalent to $2\omega \omega'' > (\omega')^2$. But
\begin{align*}
	2\omega(a) \omega''(a) - (\omega'(a))^2 & = 4 d_2(d_0+d_1a+d_2a^2) -  (2d_2 a+d_1)^2 = 4d_2d_0-d_1^2 \\
	& = 32\lambda p^2(1-p)^2 - 64\lambda^2 p^2(1-p)^2 = 32 \lambda p^2(1-p)^2(1-2\lambda)>0.
\end{align*}

We would like to finish the paper with a list of open questions that arose during our study.

\begin{que*}
Let us assume that the measure $\mu$ has an even log-concave density (not-necessarily product).
\begin{itemize}
\item Does the assertion of Theorem \ref{BMgen} holds true for arbitrary symmetric sets $A$ and $B$?
\item If not, is it true under additional assumption that the measure is product? \item  In particular, can one remove the assumption of unconditionality in the Gaussian Brunn-Minkowski inequality?
\end{itemize}
\end{que*}

\section*{Acknowledgements}

We would like to thank Rafa\l \ Lata\l a and Tomasz Tkocz for reading an early version of this article. The third named author would like to acknowledge the hospitality of the Department of Mathematical Sciences, Kent State University, where part of this work was conducted.

\vspace{1cm}

\noindent Galyna Livshyts \\
Department of Mathematical Sciences \\
Kent State University \\
Kent, OH 44242, USA \\
E-mail address: glivshyt@kent.edu

\vspace{0.8cm}

\noindent Arnaud Marsiglietti \\
Institute for Mathematics and its Applications \\
University of Minnesota \\
207 Church Street, 434 Lind Hall, \\
Minneapolis, MN 55455, USA \\
E-mail address: arnaud.marsiglietti@ima.umn.edu

\vspace{0.8cm}

\noindent Piotr Nayar \\
Institute for Mathematics and its Applications \\
University of Minnesota \\
207 Church Street, 432 Lind Hall, \\
Minneapolis, MN 55455, USA \\
E-mail address: nayar@ima.umn.edu

\vspace{0.8cm}

\noindent Artem Zvavitch \\
Department of Mathematical Sciences \\
Kent State University \\
Kent, OH 44242, USA \\
E-mail address: zvavitch@math.kent.edu


\begin{thebibliography}{9}


\bibitem{Ba} F. Barthe, N. Huet, \emph{On Gaussian Brunn-Minkowski inequalities}, Studia Math. 191 (2009),
283-304.

\bibitem{B} C. Borell, \emph{Convex set functions in d-space}, Period. Math. Hungarica 6 (1975), 111–136.

\bibitem{B2} C. Borell, \emph{The Brunn-Minkowski inequality in Gauss space}, Invent. Math. 30, 2 (1975), 207–216.

\bibitem{B3} C. Borell,
\emph{The Ehrhard inequality},
C. R. Math. Acad. Sci. Paris 337 (2003), no. 10, 663–666.

\bibitem{BLYZ1} K. J. B\"or\"oczky, E. Lutwak, D. Yang, G. Zhang, \emph{The log-Brunn-Minkowski inequality}, Adv. Math. 231 (2012), 1974-1997.

\bibitem{BLYZ2}  K. J. B\"or\"oczky, E. Lutwak, D. Yang, G. Zhang, \emph{The logarithmic Minkowski problem}, J. Amer. Math. Soc. 26
(2013), 831–852.


\bibitem{CFM} D. Cordero-Erausquin, M. Fradelizi, B. Maurey, \emph{The (B) conjecture for the Gaussian
measure of dilates of symmetric convex sets and related problems}, J. Funct. Anal. 214 (2004), no. 2, 410–427.

\bibitem{C} M. Costa, T. M. Cover, \emph{On the similarity of the entropy power inequality
and the Brunn-Minkowski inequality}, IEEE Trans. Inform. Theory 30 (1984), no. 6, 837--839.

\bibitem{E} A. Ehrhard, \emph{Sym\'etrisation dans l'espace de Gauss}, Math. Scand. 53 (1983), 281–301.

\bibitem{FM} M. Fradelizi, A. Marsiglietti, \emph{On the analogue of the concavity of entropy power in the Brunn-Minkowski theory},
Adv. in Appl. Math. 57 (2014), 1–20.

\bibitem{G1}  R. J. Gardner, \emph{The Brunn-Minkowski inequality}, Bull. Amer. Math. Soc. (N.S.) 39 (2002), no. 3, 355--405.

\bibitem{G2}  R.  J.~Gardner, {\em Geometric tomography.} Second edition. Encyclopedia of
Mathematics and its Applications, 58 Cambridge University
Press, Cambridge, 2006.


\bibitem{GZ}  R. J. Gardner, A. Zvavitch,  \emph{Gaussian Brunn-Minkowski inequalities}, Trans. Amer. Math. Soc. 362 (2010), no. 10, 5333–5353.

\bibitem{KS} S. Kwapien, J. Sawa,
\emph{On some conjecture concerning Gaussian measures of dilatations of convex symmetric sets},
Studia Math. 105 (1993), no. 2, 173–187.

\bibitem{L1} R. Latała, \emph{A note on the Ehrhard inequality}, Studia Math. 118 (1996) 169–174.

\bibitem{L2} R. Lata\l a, \emph{On some inequalities for Gaussian measures}, Proceedings of the International Congress of Mathematicians, Vol. II (Beijing, 2002), 813-822, Higher Ed. Press, Beijing, 2002.

\bibitem{Li}  A. Livne  Bar-on,
\emph{The (B) conjecture   for   uniform   measures   in   the   plane},
Geometric Aspects of Functional Analysis,
Lecture Notes in Mathematics Volume 2116, 2014, 341--353.

\bibitem{LO} R. Lata\l a, K. Oleszkiewicz, \emph{Gaussian measures of dilatations of convex symmetric sets}, Ann. Probab. 27 (1999), 1922--1938.

\bibitem{LO2} R. Lata\l a and K. Oleszkiewicz, \emph{Small ball probability estimates in terms of widths}, Studia Math. 169 (2005), 305--314.

\bibitem{Le} L. Leindler, \emph{On a certain converse of H\"older’s inequality}, II, Acta Sci. Math. (Szeged) 33
(1972), 217–223.

\bibitem{M} A. Marsiglietti, \emph{Concavity properties of extensions of the parallel
volume}, Mathematika (2015), Available on CJO2015 doi:10.1112/S0025579314000369.

\bibitem{M2} A. Marsiglietti, \emph{On the improvement of concavity of convex measures}, Proc. Amer. Math. Soc. (2015), to appear, arXiv:1403.7643.

\bibitem{NT} P. Nayar, T. Tkocz, \emph{A note on a Brunn-Minkowski inequality for the Gaussian measure},
Proc. Amer. Math. Soc. 141 (2013), no. 11, 4027–4030.

\bibitem{NT2} P. Nayar, T. Tkocz, \emph{S-inequality for certain product measures}, Math. Nachr. 287 (2014), no. 4, 398–404.

\bibitem{NT3} P. Nayar, T. Tkocz,
\emph{The unconditional case of the complex S-inequality},
Israel J. Math. 197 (2013), no. 1, 99–106.

\bibitem{P} A. Pr\'ekopa, \emph{On logarithmic concave measures and functions}, Acta Sci. Math. 34 (1973), 335–343.

\bibitem{S1} C. Saroglou, \emph{Remarks on the conjectured log-Brunn-Minkowski inequality}, 2014, Geom. Dedicata (to appear), arXiv:1311.4954.

\bibitem{S2} C. Saroglou, \emph{More on logarithmic sums of convex bodies}, arXiv:1409.4346.

\bibitem{Sn} R. Schneider, \emph{Convex bodies: the Brunn-Minkowski theory}, Second expanded edition. Encyclopedia of Mathematics and its Applications, 151. Cambridge University Press, Cambridge, 2014.

\bibitem{ST} V. N. Sudakov, B. S. Cirel'son,  \emph{Extremal properties of half-spaces for spherically invariant
measures}, Problems in the theory of probability distributions, II.
Zap. Naucn. Sem. Leningrad. Otdel. Mat. Inst. Steklov. (LOMI) 41 (1974), 165, 14–24.

\end{thebibliography}
\end{document}